\documentclass[11pt]{article}
\usepackage{}
\usepackage{cite}
\usepackage{mathrsfs}
\usepackage{amsfonts}
\usepackage{amsmath}
\usepackage{amsfonts,amssymb,color}
\usepackage{dsfont}
\usepackage{graphicx}
\usepackage{curves}
\usepackage{mathrsfs}
\usepackage{pifont}
\usepackage{amssymb}

\newtheorem{theorem}{Theorem}[section]
\newtheorem{proposition}{Proposition}[section]
\newtheorem{definition}{Definition}[section]
\newtheorem{lemma}{Lemma}[section]
\newtheorem{corollary}{Corollary}[section]
\newtheorem{remark}{Remark}[section]

\newtheorem{conjecture}{Conjecture}[section]

 \numberwithin{equation}{section}
\numberwithin{figure}{section}
\DeclareMathOperator{\vol}{vol}

\usepackage{amsmath}

\newenvironment{proof}{
\noindent {\hspace*{0.7cm}\bf Proof.}\rm}
{\mbox{}\hfill\rule{0.5em}{0.809em}\par}

\begin{document}

\title{A unified combinatorial view beyond some spectral properties}

\author{Xiaofeng Gu\thanks{
Department of Computing and Mathematics, University of West Georgia, Carrollton, GA 30118, USA. \newline Email: {\tt xgu@westga.edu}
}\ \
and
Muhuo Liu\thanks{
Department of Mathematics and Research Center for Green Development of Agriculture, South China Agricultural University, Guangzhou, 510642, China.
Email: {\tt liumuhuo@163.com} (Corresponding author)}
}

\date{}
\maketitle
 
\begin{abstract}
Let $\beta>0$. Motivated by jumbled graphs defined by Thomason, the celebrated expander mixing lemma and Haemers's vertex separation inequality, we define that a graph $G$ with $n$ vertices is a weakly $(n,\beta)$-graph if $\frac{|X| |Y|}{(n-|X|)(n-|Y|)} \le \beta^2$ holds for every pair of disjoint proper subsets $X, Y$ of $V(G)$ with no edge between $X$ and $Y$, and it is an $(n,\beta)$-graph if in addition $X$ and $Y$ are not necessarily disjoint. Our main results include the following.

(i) For any weakly $(n,\beta)$-graph $G$, the matching number
$\alpha'(G)\ge \min\left\{\frac{1-\beta}{1+\beta},\, \frac{1}{2}\right\}\cdot (n-1).$
If in addition $G$ is a $(U, W)$-bipartite graph with $|W|\ge t|U|$ where $t\ge 1$, then $\alpha'(G)\ge \min\{t(1-2\beta^2),1\}\cdot |U|$.

(ii)
For any $(n,\beta)$-graph $G$, $\alpha'(G)\ge \min\left\{\frac{2-\beta}{2(1+\beta)},\, \frac{1}{2}\right\}\cdot (n-1).$ If in addition $G$ is a $(U, W)$-bipartite graph with $|W|\ge |U|$ and no isolated vertices, then $\alpha'(G)\ge \min\{1/\beta^{2},1\}\cdot |U|$.

(iii)
If $G$ is a weakly $(n,\beta)$-graph for $0<\beta\le 1/3$ or an $(n,\beta)$-graph for $0<\beta\le 1/2$, then $G$ has a fractional perfect matching. In addition, $G$ has a perfect matching when $n$ is even and $G$ is factor-critical when $n$ is odd.

(iv) For any connected $(n,\beta)$-graph $G$, the toughness $t(G)\ge \frac{1-\beta}{\beta}$. For any connected weakly $(n,\beta)$-graph $G$, $t(G)> \frac{5(1-\beta)}{11\beta}$ and if $n$ is large enough, then $t(G) >\left(\frac{1}{2}-\varepsilon\right)\frac{1-\beta}{\beta}$ for any $\varepsilon >0$.

The results imply many old and new results in spectral graph theory, including several new lower bounds on matching number, fractional matching number  and toughness from eigenvalues. In particular, we obtain a new lower bound on toughness via normalized Laplacian eigenvalues that extends a theorem originally conjectured by Brouwer from regular graphs to general graphs.
\end{abstract}


\hspace{1.5pt}\small \noindent {\bf Key words:} expander mixing lemma, eigenvalue, matching, fractional matching, toughness, Hamilton cycle

\section{Introduction}
Throughout this paper, we only consider nonempty simple graphs and $G$ is a graph with $n$ vertices. Let $\delta$ and $\Delta$ be the minimum degree and maximum degree of $G$, respectively. For two vertex subsets $X$ and $Y$ of $G$, $e(X, Y)$ denotes the number of edges with one end in $X$ and the other one in $Y$ (edges with both ends in $X\cap Y$ are counted twice). 
We use $\lambda_i:=\lambda_i(G)$ to denote the $i$-th largest eigenvalue of the adjacency matrix of $G$. By the Perron-Frobenius Theorem, $\lambda_1$ is always positive and $|\lambda_i|\le \lambda_1$ for all $i\ge 2$. Let $\lambda = \max_{2\le i\le n} |\lambda_i|$, that is, the second largest absolute eigenvalue. 
For a $d$-regular graph, it is known that $\lambda_1=d$. A $d$-regular graph on $n$ vertices with the second largest absolute eigenvalue at most $\lambda$ is called an {\bf $(n, d, \lambda)$-graph}. The celebrated expander mixing lemma states that for any $(n, d, \lambda)$-graph and two vertex subsets $X$ and $Y$, $$ \left| e(X, Y) - \frac{d}{n} |X| |Y| \right| \le \lambda \sqrt{|X| |Y| \left(1-\frac{|X|}{n}\right)\left(1-\frac{|Y|}{n}\right)}.$$
This result is usually attributed to Alon and Chung~\cite{AlCh88}. It is also noticed that the idea appeared earlier with a different form in the PhD thesis~\cite{Haem79} of Haemers.

If $e(X, Y)=0$, the expander mixing lemma implies that for $(n, d, \lambda)$-graphs, then  we have
\begin{align}\label{Liu1e}
\frac{|X| |Y|}{(n-|X|)(n-|Y|)} \le \left(\frac{\lambda}{d}\right)^2.
\end{align}
 When $X,Y$ are disjoint vertex subsets with $e(X, Y)=0$, Haemers~\cite{Haem95} showed that
\begin{align}\label{Liu2e}
\frac{|X| |Y|}{(n-|X|)(n-|Y|)} \le \left(\frac{\mu_n - \mu_2}{\mu_n + \mu_2}\right)^2
\end{align}
for any graph $G$, where $\mu_i$ denotes the $i$-th smallest eigenvalue of the Laplacian matrix of $G$.

These inequalities are more like combinatorial properties rather than spectral properties. In fact, Thomason~\cite{Thom85, Thom87} introduced the definition of jumbled graphs by using the edge density of all subgraphs, which is the first quantitative definition of pseudorandom graphs.
We observe that even we ignore the edge distribution of graphs, the vertex separations also imply nice structural properties. This observation motivates the following definitions.
\begin{definition}
Let $\beta>0$.
\par\noindent
{\em(i)} A graph $G$ with $n$ vertices is a {\bf weakly $(n,\beta)$-graph} if
\begin{align}\label{Liu3e}\frac{|X| |Y|}{(n-|X|)(n-|Y|)} \le \beta^2\end{align}
holds for every pair of disjoint proper subsets $X, Y$ of $V(G)$ with no edge between $X$ and $Y$.
\par\noindent
{\em(ii)} It is an {\bf $(n,\beta)$-graph} if \eqref{Liu3e} holds for every pair of proper subsets $X, Y$ of $V(G)$ with no edge between $X$ and $Y$.
\end{definition}

By convention, the complete graph $K_n$ is an $(n,0)$-graph, which means that $K_n$ is an $(n,\beta)$-graph for every $\beta>0$.  Clearly any $(n,\beta)$-graph is also a weakly $(n,\beta)$-graph. Notice that if $\beta \ge 1$, every graph is a weakly $(n,\beta)$-graph trivially. Thus it is more interesting when $0<\beta < 1$. By~\eqref{Liu1e}, every $(n, d, \lambda)$-graph is an $(n,\lambda/d)$-graph, and by~\eqref{Liu2e}, every graph is a weakly $(n,\beta)$-graph where $\beta\ge \frac{\mu_n - \mu_2}{\mu_n + \mu_2}$. In the next section, we introduce more $(n,\beta)$-graphs and weakly $(n,\beta)$-graphs.

We prove two general separation results (Lemmas~\ref{22l} and~\ref{L23l}),  through which we will study several properties of (weakly) $(n,\beta)$-graphs, including matching, fractional matching, factor-critical graph, and toughness, etc. These results imply many old and new results in spectral graph theory.

The rest of this paper is organized as below.
In the next section, we list $(n,\beta)$-graphs and weakly $(n,\beta)$-graphs via eigenvalues of adjacency matrix, Laplacian matrix and normalized Laplacian matrix. We study matching number and toughness of (weakly) $(n,\beta)$-graphs in Sections~\ref{sect:matching} and~\ref{sect:toughness}, respectively. By applying these results for specific values of $\beta$, we obtain many results involving eigenvalues in each corresponding section. In Section~\ref{sect:bipart}, a bipartite analogue of $(n,\beta)$-graphs will be introduced.

\section{$(n,\beta)$-graphs and weakly $(n,\beta)$-graphs via eigenvalues}

In this section, we show that every graph is an $(n,\beta)$-graph and/or a weakly $(n,\beta)$-graph for some value $\beta$ involving eigenvalues. These are listed in Proposition~\ref{prop:ex}.

Recall that the {\bf Laplacian matrix} of a graph $G$ is the matrix $L=D-A$, where $D$ is the diagonal matrix of vertex degrees  and $A$ is the adjacency matrix of $G$; and $\mu_i:=\mu_i(G)$ denotes the $i$-th smallest eigenvalue of the Laplacian matrix of $G$ for $i=1,2,\ldots, n$. We have $0=\mu_1\le \mu_2\le \cdots \le \mu_n$. According to~\cite{Chun92}, the {\bf normalized Laplacian matrix $\mathcal{L}$} of $G$ is an $n\times n$ matrix such that the $(u,v)$-entry equals $-1/\sqrt{d(u)d(v)}$ if $u$ is adjacent to $v$, equals 1 if $u=v$ and $d(v)\neq 0$, and 0 otherwise.  
Let $\sigma_i:=\sigma_i(G)$ denote the $i$-th smallest eigenvalue of the normalized Laplacian matrix of $G$, and it is known that $0=\sigma_1\le \sigma_2\le \cdots \le\sigma_n\le 2$. Let $\sigma = \max_{2\le i\le n}\{|1-\sigma_i|\}$. More about normalized Laplacian matrix and its eigenvalues can be found in~\cite{Chun92, Chun04} by Chung.

\begin{proposition}\label{prop:ex}
We have the following.
\\{\em(i)} Every $(n, d, \lambda)$-graph is an $(n,\beta)$-graph, where $\beta\ge \lambda/d$.
\\{\em(ii)} Every graph is a weakly $(n,\beta)$-graph, where $\beta\ge \frac{\mu_n - \mu_2}{\mu_n + \mu_2}$.
\\{\em(iii)} If $2\delta \ge \mu_2 +\mu_n$ for a graph $G$, then $G$ is an $(n,\beta)$-graph, where $\beta \ge \frac{\mu_n - \mu_2}{\mu_n + \mu_2}$.
\\{\em(iv)} Every graph with no isolated vertices is an $(n,\beta)$-graph, where $ \beta\ge \sigma\Delta/\delta$.
\\{\em(v)} Every connected  graph is a weakly $(n,\beta)$-graph, where $ \beta\ge \frac{\sigma_n-\sigma_2}{\sigma_n+\sigma_2}\cdot \frac{\Delta}{\delta}$.
\end{proposition}
\begin{proof} It is straightforward to see that
(i) follows from \eqref{Liu1e},  (ii) follows from \eqref{Liu2e}, (iii) follows from Corollary~\ref{cor:fcn}, (iv) follows from Corollary~\ref{cor:1t} and  (v) follows from Corollary~\ref{cor:2t}.\end{proof}   

\begin{remark}
For any graph,
$\displaystyle \frac{\sigma_n-\sigma_2} {\sigma_n+\sigma_2}\le \sigma.$
\end{remark}
\begin{proof}
For complete graphs, it is known from~\cite{Chun92} that $\sigma_2=\sigma_n =n/(n-1)$, and thus $\displaystyle \frac{\sigma_n-\sigma_2} {\sigma_n+\sigma_2}\le \sigma$.
For noncomplete graphs, we have $\sigma_2\le 1 <\sigma_n$ by~\cite{Chun92}. Thus $\sigma_2 -1\le 0 <\sigma_n-1$ and so $\sigma=\max\{\sigma_n-1,1-\sigma_2\}$.

If $\sigma_n + \sigma_2\le 2$, that is, $\sigma_n -1\le 1-\sigma_2$, then $\sigma = 1-\sigma_2$. It follows that
$$\frac{\sigma_n-\sigma_2}{\sigma_n+\sigma_2} = 1-\frac{2\sigma_2}{\sigma_n+\sigma_2}\le 1 - \sigma_2 =\sigma.$$

If $\sigma_n + \sigma_2\ge 2$, that is, $\sigma_n -1\ge 1-\sigma_2$, then $\sigma = \sigma_n -1$. It follows that
$$\frac{\sigma_n-\sigma_2}{\sigma_n+\sigma_2} = \frac{2\sigma_n}{\sigma_n+\sigma_2}-1\le \sigma_n -1 =\sigma,$$
as desired.
\end{proof}

The following two theorems of Chung~\cite{Chun04} are generalizations of the celebrated expander mixing lemma to Laplacian eigenvalues and normalized Laplacian eigenvalues.
\begin{theorem}[Chung~\cite{Chun04}]\label{thm:fc2}
Let $d'=(\mu_n +\mu_2)/2$. For any two subsets $X$ and $Y$ of $V(G)$,
$$\left| e(X, Y) - \frac{d'}{n} |X| |Y|  + d' |X\cap Y|  - \sum_{v\in X\cap Y} d_G(v) \right| \le \frac{\mu_n -\mu_2}{2} \sqrt{|X| |Y| \left(1- \frac{|X|}{n}\right) \left(1- \frac{|Y|}{n}\right)}.$$
\end{theorem}

\begin{corollary}\label{cor:fcn}
If $\delta\ge (\mu_n +\mu_2)/2$, then for any two subsets $X$ and $Y$ of $V(G)$ with no edge between $X$ and $Y$,
\begin{equation*}
\frac{|X| |Y|}{(n-|X|)(n-|Y|)} \le \left(\frac{\mu_n - \mu_2}{\mu_n + \mu_2}\right)^2.
\end{equation*}
\end{corollary}
\begin{proof}
If $\delta\ge (\mu_n +\mu_2)/2 =d'$, then $\sum_{v\in X\cap Y} d_G(v)\ge d' |X\cap Y|$.
By Theorem~\ref{thm:fc2},
$$\frac{d'}{n} |X| |Y|\le \left|0 - \frac{d'}{n} |X| |Y|  + d' |X\cap Y|  - \sum_{v\in X\cap Y} d_G(v) \right| \le \frac{\mu_n -\mu_2}{2} \sqrt{|X| |Y| \left(1- \frac{|X|}{n}\right) \left(1- \frac{|Y|}{n}\right)}.$$
It follows that
$$\left(\frac{d'}{n} |X| |Y|\right)^2\le \left(\frac{\mu_n -\mu_2}{2}\right)^2 |X| |Y| \left(1- \frac{|X|}{n}\right) \left(1- \frac{|Y|}{n}\right),$$
and so
$$(d')^2 |X| |Y|\le \left(\frac{\mu_n -\mu_2}{2}\right)^2 \left(n-|X|\right) \left(n- |Y|\right),$$
implying that
$$\frac{|X| |Y|}{(n-|X|)(n-|Y|)} \le \left(\frac{\mu_n -\mu_2}{2d'}\right)^2 = \left(\frac{\mu_n - \mu_2}{\mu_n + \mu_2}\right)^2,$$
as desired.\end{proof}

For a graph $G$ and a subset $X\subseteq V(G)$, let $\overline{X}=V(G)\setminus X$ and define the volume $\vol(X)$ as $$\vol(X)=\sum_{v\in X} d(v).$$
\begin{theorem}[Chung~\cite{Chun04}]\label{thm:fc}
For a graph $G$ and two subsets $X$ and $Y$ of $V(G)$,
$$\left| e(X, Y) - \frac{\vol(X)\vol(Y)}{\vol(G)} \right| \le \frac{\sigma \sqrt{\vol(X)\vol(\overline{X})\vol(Y)\vol(\overline{Y})}}{\vol(G)}.$$
\end{theorem}

\begin{corollary}\label{cor:1t}
Suppose that $X$ and $Y$ are two subsets of $V(G)$ such that there is no edge between $X$ and $Y$. If $G$ contains no isolated vertices, then
\begin{equation*}
\frac{|X| |Y|}{(n-|X|)(n-|Y|)} \le \left(\frac{\sigma\Delta}{\delta}\right)^2.
\end{equation*}
\end{corollary}
\begin{proof}
By Theorem~\ref{thm:fc},
$$\left| 0 - \frac{\vol(X)\vol(Y)}{\vol(G)} \right| \le \frac{\sigma \sqrt{\vol(X)\vol(\overline{X})\vol(Y)\vol(\overline{Y})}}{\vol(G)},$$
and so
$$\vol(X)\vol(Y) \le \sigma \sqrt{\vol(X)\vol(\overline{X})\vol(Y)\vol(\overline{Y})}.$$
It follows that
$$\vol(X)\vol(Y) \le \sigma^2 \vol(\overline{X})\vol(\overline{Y}).$$

Notice that $\delta|X|\le \vol(X)$, $\delta|Y|\le \vol(Y)$, $\vol(\overline{X})\le \Delta|\overline{X}|=\Delta(n-|X|)$ and $\vol(\overline{Y})\le \Delta|\overline{Y}|=\Delta(n-|Y|)$, we have
$$\delta^2 |X||Y|\le \sigma^2 \Delta^2 (n-|X|)(n-|Y|),$$
and thus it follows.
\end{proof}

\begin{theorem}[Butler~\cite{Bu08}]\label{n2t}
Let $G$ be a connected graph on $n\ge 2$ vertices. If $X$ and $Y$ are two disjoint subsets of $V(G)$ such that there is no edge between $X$ and $Y$, then
\begin{equation*}
\frac{\vol(X)\vol(Y)}{\vol(\overline{X})\vol(\overline{Y})} \le \left(\frac{\sigma_n-\sigma_2}{\sigma_n+\sigma_2}\right)^2.
\end{equation*}
\end{theorem}

\begin{corollary}\label{cor:2t}
Let $G$ be a connected graph on $n\ge 2$ vertices.
If $X$ and $Y$ are two disjoint subsets of $V(G)$ such that there is no edge between $X$ and $Y$, then
$$\frac{|X||Y|}{(n-|X|)(n-|Y|)}\le \left(\frac{\sigma_n-\sigma_2}{\sigma_n+\sigma_2}\cdot \frac{\Delta}{\delta}\right)^2.$$
\end{corollary}
\begin{proof}
Similar to the proof of Corollary~\ref{cor:1t}.
\end{proof}

\section{Matching number}\label{sect:matching}
A {\bf matching} in a graph is a set of pairwise non-adjacent  edges. A {\bf maximum matching} is one with maximum size among all matchings of the graph. The {\bf matching number} of a graph $G$, denoted by $\alpha'(G)$, is the size of a maximum matching in $G$.
A {\bf perfect matching} in a graph $G$ with $|V(G)|$ even is a matching of size $|V(G)|/2$.
It has been shown by Tutte~\cite{Tutte47} that a graph $G$ has a perfect matching if and only if $o(G-S)\le |S|$ for every $S\subseteq V(G)$, where $o(G-S)$ denotes the number of odd components of $G-S$. Generalizing Tutte's condition, Berge discovered the following formula for the matching number. The formula can also be derived from Tutte's theorem, and thus is called the {\bf Berge-Tutte Formula}.

\begin{theorem}[Berge~\cite{Berge58}]
\label{btfo}
The matching number of a graph $G$ is $$\alpha'(G) = \frac{1}{2} \left( |V(G)| + \min_{S\subseteq V(G)} \big\{|S| - o(G-S)\big\} \right).$$
\end{theorem}

Here is an analogue of the Berge-Tutte Formula for bipartite graphs by Ore, which can also be obtained by the well-known Hall's condition~\cite{Hall35}. A $(U,W)$-bipartite graph is a bipartite graph with bipartition $U, W$. For a vertex subset $S$, let $N(S)$ denote the set of vertices having a neighbor in $S$.
\begin{theorem}[Ore~\cite{Ore55}]
\label{hall}
The matching number of a $(U, W)$-bipartite graph $G$ is
$$\alpha'(G) = |U| + \min_{S\subseteq U} \big\{|N(S)|-|S|\big\}.$$
\end{theorem}

We prove lower bounds on matching number of (weakly) $(n,\beta)$-graphs.

\begin{theorem}\label{main1}
Let $G$ be a weakly $(n,\beta)$-graph. Then
$$\alpha'(G)\ge \min\left\{\frac{1-\beta}{1+\beta},\, \frac{1}{2}\right\}\cdot (n-1).$$
In particular, if $n$ is even and $0<\beta\le 1/3$, then $G$ has a perfect matching.
\end{theorem}

\begin{theorem}\label{main2}
Let $G$ be an $(n,\beta)$-graph. Then
$$\alpha'(G)\ge \min\left\{\frac{2-\beta}{2(1+\beta)},\, \frac{1}{2}\right\}\cdot (n-1).$$
In particular, if $n$ is even and $0<\beta\le 1/2$, then $G$ has a perfect matching.
\end{theorem}

\begin{theorem}\label{main3}
Let $G$ be a $(U, W)$-bipartite graph with $|W|\ge t|U|$, where $t\ge 1$. If $G$ is a weakly $(n,\beta)$-graph, then $$\alpha'(G)\ge \min\left\{t(1-2\beta^2),1\right\}\cdot |U|.$$
\end{theorem}

\begin{theorem}\label{main4}
Let $G$ be a $(U, W)$-bipartite graph with $|W|\ge |U|$. If $G$ is an $(n,\beta)$-graph with no isolated vertices, then $$\alpha'(G)\ge \min\left\{1/\beta^{2},1\right\}\cdot |U|.$$
\end{theorem}

A graph $G$ is {\bf factor-critical} if for every vertex $v\in V(G)$, $G-v$ has a perfect matching. Obviously, any factor-critical graph has an odd number of vertices. We have the following result for factor-critical graphs.
\begin{theorem}\label{main5}
Let $n$ be odd.
\\{\rm(i)} If $G$ is a weakly $(n,\beta)$-graph and $0<\beta\le 1/3$, then $G$ is factor-critical.
\\{\rm(ii)} If $G$ is an $(n,\beta)$-graph and $0<\beta\le 1/2$, then $G$ is factor-critical.
\end{theorem}

A {\bf fractional matching} of a graph $G$ is a function $f$ that assigns to each edge of $G$ a real number in $[0,1]$ so that for each vertex $v\in V(G)$, we have $\sum f(e)\le 1$ where the sum is taken over all edges incident to $v$. It is not hard to see that $\sum_{e\in E(G)} f(e)\le |V(G)|/2$ for any fractional matching $f$ of $G$. A fractional matching $f$ of $G$ is called a {\bf fractional perfect matching} if $\sum_{e\in E(G)} f(e) = |V(G)|/2$. The {\bf fractional matching number} $\alpha'_f(G)$ of $G$ is the maximum of $\sum_{e\in E(G)} f(e)$ over all fractional matchings of $G$. Clearly $\alpha'_f(G)\ge\alpha'(G)$ for any graph $G$, however, for bipartite graphs, $\alpha'_f(G) =\alpha'(G)$. More about fractional matching can be found in~\cite[Chapter 2]{SU11}.

Clearly if a graph has a perfect matching, then it has a fractional perfect matching. It is also known that any factor-critical graph has a fractional perfect matching~\cite{CP83}. Thus, by Theorems \ref{main1}, \ref{main2} and \ref{main5}, we have the following corollary.
\begin{corollary}\label{fracpm}
If $G$ is a weakly $(n,\beta)$-graph for $0<\beta\le 1/3$ or an $(n,\beta)$-graph for $0<\beta\le 1/2$, then $G$ has a fractional perfect matching.
\end{corollary}

We would also like to mention a fractional version of the Berge-Tutte formula, which will be used later in this section.
\begin{theorem}[The fractional Berge-Tutte Formula~\cite{SU11}]\label{thm:frac}
The fractional matching number of $G$ is $$\alpha'_f(G)=\frac{1}{2} \left( |V(G)| + \min_{S\subseteq V(G)} \big\{|S| -i(G-S)\big\} \right),$$
where $i(G-S)$ is the number of isolated vertices in $G-S$. In particular, $G$ has a fractional perfect matching if and only if $i(G-S)\le |S|$ for every $S\subseteq V(G)$.
\end{theorem}

We will present the proofs of these theorems and some applications in spectral graph theory in the next two subsections, respectively.

\subsection{Proofs of the theorems}
The following lemma was shown in~\cite{GL22} for a specific value of $\beta$ involving Laplacian eigenvalues, however, the proof in~\cite{GL22} actually indicates a more general combinatorial result as below. For completeness, we include a proof.
\begin{lemma}\label{22l}
Suppose that $S\subset V(G)$ such that $G-S$ is disconnected. Let $X$ and $Y$ be disjoint vertex subsets of $G-S$ such that $X\cup Y= V(G)-S$ with $|X|\le |Y|$. If $|X| |Y| \le \beta^2 (n-|X|)(n-|Y|)$ with $\beta>0$, then
\begin{equation}
\label{xupp}
|X|\le \frac{\beta n}{1+\beta},
\end{equation}
and
\begin{equation}
\label{ssiz}
|S| \ge \frac{1 -\beta}{\beta} |X|,
\end{equation}
with each equality holding  only when $|X|=|Y|$.
\end{lemma}
\begin{proof}
Since $|X|\le |Y|$, we have
$$|X|^2 \le |X|\cdot |Y| \le \beta^2 (n-|X|)(n-|Y|)\le \beta^2 (n-|X|)^2,$$
that is $$|X| \le \beta (n - |X|),$$ and hence
\begin{equation}
\label{betax}
|X|\le \frac{\beta n}{1+\beta},
\end{equation}
with the equality holding  only when $|X|=|Y|$.

\par\medskip
Notice that \eqref{ssiz} is trivial if $\beta\ge 1$. Thus we may assume that $0<\beta<1$. Since $|Y| = n - |S| - |X|$, we have
$$ |X| (n - |S| - |X|) = |X|\cdot |Y| \le \beta^2 (n - |X|)(n - |Y|) = \beta^2 (n - |X|)(|S| + |X|),$$
implying that
\begin{equation}
\label{xnbo}
|X| n \le \left(\beta^2 (n - |X|) + |X| \right) \left(|S| + |X| \right)
= \left(\beta^2 n + (1-\beta^2)|X| \right) \left(|S| + |X| \right).
\end{equation}
By~\eqref{betax}, we have
$$(1-\beta^2)|X|\le (1-\beta^2) \cdot \frac{\beta n}{1+\beta} = (\beta - \beta^2) n,$$
which, together with \eqref{xnbo}, implies that
$$|X|n \le \left(\beta^2 n + (\beta - \beta^2) n \right) \left(|S| + |X| \right) = \beta n \left(|S| + |X| \right),$$
and we have
$$ |X| \le \beta \left(|S| + |X| \right).$$
Hence,
\begin{equation*}
|S| \ge \frac{1 -\beta}{\beta} |X|.
\end{equation*}
Since~\eqref{betax} was utilized, the equality holds in~\eqref{ssiz} only when $|X|=|Y|$.
\end{proof}

\begin{lemma}\label{L23l}
Let $S$ be a subset of $V(G)$ such that $G-S$ contains at least $c$ components, where $c\ge 2$.
\par\noindent
{\rm (i)} If $G$ is a weakly $(n,\beta)$-graph and $0<\beta<1$, then $$|S|>\frac{(c-1)(1-\beta)}{2\beta}\,\,\,\,\text{and}\,\,\,\, c-|S|<\max\left\{\frac{(3\beta-1)n +2(1-\beta)}{\beta+1},0\right\}.$$
\par\noindent
{\rm (ii)} If $G$ is an $(n,\beta)$-graph, then $$c\le \frac{\beta n}{1+\beta}\,\,\,\, \text{and}\,\,\,\, c-|S|\le \max\left\{\frac{2\beta-1}{1+\beta}\cdot n,0\right\}.$$
Furthermore, if $G-S$ contains $c-1$ isolated vertices and $|V(G)-S|\ge 2\beta n/(1+\beta)$, then
$$|S|>\frac{2(c-1)}{\beta(1+\beta)}\ge \frac{c}{\beta(1+\beta)}.$$
\end{lemma}
\begin{proof}
Let $D_1, D_2, \ldots, D_c$ be the vertex sets of the $c$ components of $G-S$. Without loss of generality, we may assume that $|D_1|\le |D_2| \le \cdots \le |D_c|$.

(i): 
   Define $X =\bigcup_{1\le i\le \lfloor c/2\rfloor} D_i$ and $Y = V(G) - S -X$. Then $\frac{c-1}{2} \le |X|\le |Y|$ and $e(X,Y)=0$. By~\eqref{ssiz} and $|X|\geq \frac{c-1}{2}$, we have
\begin{equation}\label{fc2}
|S| \ge \frac{1 -\beta}{\beta} |X|\ge \frac{(1-\beta)(c-1)}{2\beta},
\end{equation}
with equality holding only when $\frac{c-1}{2}=|X|=|Y|$ by Lemma~\ref{22l}. However, if $c$ is even, then $\frac{c-1}{2}<\frac{c}{2}\le |X|$; and
if $c$ is odd, then $|X| < |Y|$ by definitions of $X$ and $Y$. Thus the equality never hold in~\eqref{fc2}.

Next, we show the upper bound for $c-|S|$. It suffices to show that
\begin{equation}\label{ne1}
\text{ if $c-|S| \ge 0$, then $c-|S|<\frac{(3\beta-1)n +2(1-\beta)}{\beta+1}$}.
\end{equation}

We first show that if $|V(G) -S| = c$, then \eqref{ne1} holds. In this case, each $D_i$ is an isolated vertex for $1\le i\le c$. By \eqref{fc2},
$$ n-c=|S|>\frac{(c-1)(1-\beta)}{2\beta},$$
and thus
\begin{equation*}
c<\frac{\beta(2n-1)+1}{\beta+1}.
\end{equation*}
This implies that
\begin{equation*}
c-|S| =2c-n <\frac{2\beta(2n-1)+2}{\beta+1}-n=\frac{(3\beta-1)n +2(1-\beta)}{\beta+1},
\end{equation*}
and so \eqref{ne1} holds.

Thus, we may assume that $|V(G)-S| \ge c+1$ in the following. It was proved in \cite{GL22} that if $|V(G) -S|\ge c+1$, then $V(G)-S$ can be partitioned into two disjoint sets $X$ and $Y$ such that $e(X, Y)=0$ and $|Y|\ge |X|\ge c/2$. It follows that $c\le 2|X|$. Thus, by~\eqref{ssiz},
\begin{equation*}
c -|S| \le 2|X| -|S| \le 2|X| - \frac{1-\beta}{\beta} |X| =  \frac{3\beta-1}{\beta}|X|,
\end{equation*}
and thus \eqref{xupp} implies that
\begin{equation}\label{nex1}
0\le c -|S| \le \frac{3\beta-1}{\beta}|X|\le \frac{3\beta-1}{\beta} \cdot \frac{\beta}{1+\beta}\cdot n= \frac{3\beta-1}{1+\beta}\cdot n <\frac{(3\beta-1)n +2(1-\beta)}{\beta+1}.
\end{equation}

(ii): Let $U$ be a set of vertices that consists of exactly one vertex from each $D_i$ for $i=1,2,\ldots,c$. By the definition of $(n,\beta)$-graphs,
$|U|^2 \le \beta^2 (n-|U|)^2$. It follows that $|U|\le \beta(n-|U|)$ and so \begin{equation}\label{cceven}c=|U|\le \frac{\beta n}{1+\beta}.\end{equation}

Next, we show the upper bound on $c-|S|$. It suffices to show that
\begin{equation}\label{ne2}
\text{ if $c-|S| >0$, then $c-|S|\le \frac{2\beta -1}{1+\beta}\cdot n$.}
\end{equation}

Since $G$ is an $(n,\beta)$-graph, $c\le\frac{\beta n}{1+\beta}$ by~\eqref{cceven}. If $|V(G) -S|\le \frac{2\beta n}{1+\beta}$, then $|S|\ge \frac{1-\beta}{1+\beta}\cdot n$. Thus
\begin{equation*}
c -|S|\le \frac{\beta n}{1+\beta} - \frac{1-\beta}{1+\beta}\cdot n =\frac{2\beta -1}{1+\beta}\cdot n,
\end{equation*}
and so \eqref{ne2} holds. Therefore, we can assume that $|V(G) -S| > \frac{2\beta n}{1+\beta}\ge 2c$ in the following, and so $|V(G) -S|\ge 2c+1$.

{\bf Case 1:}  $\sum_{i=1}^{c-1} |D_i| = c-1$. Then each $D_i$ is a single vertex for $i=1,2,\ldots, c-1$.  Let $X= \bigcup_{i=1}^{c-1} D_i$ and $Y = V(G) -S$. Then $|X|=c-1$ and $e(X,Y)= 0$. By definition of $(n,\beta)$-graphs,
$$\frac{|X| |Y|}{(n-|X|)(n-|Y|)} \le \beta^2 \text{ and thus } \frac{|X||Y|}{\beta^2}\le (n-|X|)(n-|Y|)< n(n-|Y|)=n|S|.$$
It follows that
$$|S|> \frac{|X||Y|}{n\beta^2} = \frac{|X|}{\beta^2}\cdot \frac{|Y|}{n} \ge \frac{c-1}{\beta^2}\cdot \frac{2\beta}{1+\beta} = \frac{2(c-1)}{\beta(1+\beta)}\ge \frac{c}{\beta(1+\beta)},$$ and thus \eqref{cceven} implies that
$$c-|S|< c- \frac{c}{\beta(1+\beta)} = \frac{\beta^2+\beta-1}{\beta(1+\beta)}\cdot c\le \frac{\beta^2+\beta-1}{\beta(1+\beta)}\cdot \frac{\beta n}{1+\beta} \le\frac{2\beta^2+\beta-1}{(1+\beta)^2}\cdot n = \frac{2\beta -1}{1+\beta}\cdot n,$$
as required.

 \medskip {\bf Case 2:} $\sum_{i=1}^{c-1} |D_i|\ge c$.  Let $V_i =D_i$ for $i=1,2,\ldots,c-1$ and $V_c=D-X-S$. Thus $\sum_{i=1}^{c-1} |V_i|\ge c$ and $V(G)-S=\bigcup_{i=1}^{c} V_i$.

It was proved in \cite{Gu21b} that if $|V(G) -S|\ge 2c+1$ and $\sum_{i=1}^{c-1} |V_i|\ge c$, then $V_1,V_2,\ldots, V_c$ can be partitioned into two disjoint sets $X$ and $Y$ such that $e(X, Y)=0$ and  $|Y|\ge |X|\ge c$.
Thus, by~\eqref{ssiz},
\begin{equation}\label{eq:nb}
0 < c -|S| \le |X| -|S| \le |X| - \frac{1-\beta}{\beta} |X| =  \frac{2\beta-1}{\beta}|X|.
\end{equation}

By~\eqref{xupp} and~\eqref{eq:nb},
$$0< c -|S| \le  \frac{2\beta-1}{\beta}|X|\le \frac{2\beta-1}{\beta} \cdot \frac{\beta}{1+\beta}\cdot n= \frac{2\beta-1}{1+\beta}\cdot n,$$
 and thus  \eqref{ne2} also holds.
 \end{proof}

\begin{proof}[\bf Proof of Theorem~\ref{main1}]
Let $r=\min\left\{\frac{1-\beta}{1+\beta},\, \frac{1}{2}\right\}$. The result is trivial for $\beta\ge 1$, and so we may assume that $0<\beta<1$.
By Theorem~\ref{btfo}, it suffices to show that for every $S\subseteq V(G)$,
$$|S| - o(G-S)\ge (2r-1)n -2r,$$ i.e. $$o(G-S) - |S|\le (1-2r)n +2r.$$

If $o(G-S)\le 1$, then $o(G-S) - |S| \le 1- |S| \le 1 \le (1-2r)n +2r$, since $0\le1-2r\le (1-2r)n$ as $r\le 1/2$. Furthermore, the result clearly holds for $o(G-S)-|S|\le 0$. Thus we may assume that $c=o(G-S)\ge 2$ and $c-|S|=o(G)-|S|>0$.
Since $r\le \frac{1-\beta}{1+\beta}$, by Lemma~\ref{L23l}, we have
\begin{equation*}
0 < o(G-S)-|S| = c-|S| \le \frac{(3\beta-1)n +2(1-\beta)}{\beta+1}\le (1-2r)n +2r,
\end{equation*}
as desired.
\end{proof}

\begin{proof}[\bf Proof of Theorem~\ref{main2}]
Let $r= \min\left\{\frac{2-\beta}{2(1+\beta)},\, \frac{1}{2}\right\}$.
By Theorem~\ref{btfo}, it suffices to show that $$o(G-S) -|S|\le (1-2r)n +2r$$
for every $S\subseteq V(G)$.

If $o(G-S)\le 1$, then $o(G-S) -|S| \le 1- |S| \le 1 \le (1-2r)n +2r$, since $1-2r\le (1-2r)n$ as $r\le 1/2$. Thus we may assume that $c=o(G-S)\ge 2$ and $c-|S|>0$.
Since $r\le \frac{2-\beta}{2(1+\beta)}$,
\begin{equation*}
0<o(G-S)-|S| = c -|S|\le  \frac{2\beta -1}{1+\beta}\cdot n \le (1-2r)n,
\end{equation*}
by Lemma~\ref{L23l}.
\end{proof}

\begin{proof}[\bf Proof of Theorem~\ref{main3}]
Let $r=\min\{t(1-2\beta^2),1\}$. Obviously we can assume that $1-2\beta^2>0$.
Let $|U|=n_1$ and $|W|=n_2$, and so $n_2\ge tn_1$. By Theorem~\ref{hall}, it suffices to show that for every $S\subseteq U$, $$|S|-|N(S)|\le (1-r)n_1.$$
It is trivial if either $|S|\le (1-r)n_1$ or $|S|\le |N(S)|$ since $r\le 1$. Thus we may assume that $|S|>(1-r)n_1$ and $|S|>|N(S)|$.

By contradiction, we suppose that there exists an $S\subseteq U$ such that $|S|-|N(S)| > (1-r)n_1$. Let $|S|=x$ and $|N(S)|=y$, and so $x-y> (1-r)n_1$. Choose $X=S$ and $Y=V(G)\setminus (S\cup N(S))$. Clearly $X,Y$ are disjoint and there is no edge between $X$ and $Y$. By the definition of weakly $(n,\beta)$-graphs, we have
$$\frac{1}{2}\left(1-\frac{y}{n_1+n_2-x}\right)<\frac{x(n_1+n_2-x-y)}{(x+y)(n_1+n_2-x)}\le   \beta^2,$$ as  $y<x$.
This implies that  $y> (n_1+n_2-x)(1-2\beta^2).$

Now, we have $(1-r)n_1<x-y<x-(n_1+n_2-x)(1-2\beta^2)\le 2(1-\beta^2)n_1-(n_1+n_2)(1-2\beta^2)\le 2(1-\beta^2)n_1-(1+t)n_1(1-2\beta^2)$, which implies that
$r>t(1-2\beta^2)$, a contradiction.
\end{proof}

\begin{proof}[\bf Proof of Theorem~\ref{main4}]
Let $r=\min\{1/\beta^{2},1\}$. Let $|U|=n_1$ and $|W|=n_2$, and so $n_2\ge n_1$. By Theorem~\ref{hall}, it suffices to show that for every $S\subseteq U$,   $$|S|-|N(S)|\le (1-r)n_1.$$
This is trivial if either $|S|\le (1-r)n_1$ or $|S|\le |N(S)|$ since $r\le 1$. Thus we may assume that $|S|>(1-r)n_1$ and $|S|>|N(S)|$.

By contradiction, we suppose that there exists an $S\subseteq U$ such that $|S|-|N(S)| > (1-r)n_1$. Let $|S|=x$ and $|N(S)|=y$, and so $x-y> (1-r)n_1$. Since $G$ has no isolated vertices, $N(S)$ is nonempty.
Choose $X=S$ and $Y=U\cup (W\setminus N(S))$. Then there is no edge between $X$ and $Y$. By the definition of $(n,\beta)$-graphs, we have
$$\frac{x}{y}<\frac{x(n_1+n_2-y)}{y(n_1+n_2-x)}\le \beta^2,$$ as $y<x$.
This implies that $y>x/\beta^2.$

Now, since $r\le 1$, we have
$$0\le (1-r)n_1<x-y<\left(1-\frac{1}{\beta^2}\right)x\le\left(1-\frac{1}{\beta^2}\right)n_1,$$ which implies that $r>1/\beta^{2}$, a contradiction.
\end{proof}

\begin{proof}[\bf Proof of Theorem~\ref{main5}]
Gallai~\cite{Gall63} proved that $G$ is factor-critical if and only if $|V(G)|$ is odd and $o(G-S)\le |S|$ for every nonempty subset $S\subseteq V(G)$.  Suppose to the contrary that $G$ is not factor-critical. Then there exists a nonempty subset $S\subseteq V(G)$ such that
\begin{equation}
\label{fc1}
o(G-S)\ge |S| +1 \ge 2.
\end{equation}
\par
(i) Let $O_1, O_2, \ldots, O_c$ be the vertex sets of the odd components of $G-S$, where $c=o(G-S)$.
Since $\beta\le \frac{1}{3}$, by Lemma~\ref{L23l},
\begin{equation*}
|S| > \frac{(1-\beta)(c-1)}{2\beta}  \ge c-1 = o(G-S) -1,
\end{equation*}
contrary to~\eqref{fc1}.

\medskip
(ii) Since $c=o(G-S)\ge 2$, $c-|S|=o(G-S)-|S|>0$ and $\beta\le 1/2$, by Lemma~\ref{L23l}, we have \begin{equation*}
o(G-S)-|S| = c -|S|\le \frac{2\beta -1}{1+\beta}\cdot n\le 0,
\end{equation*}
contrary with \eqref{fc1}. This completes the proof.
\end{proof}

\subsection{Lower bounds on matching number via eigenvalues}
Matching and matching number have been studied, among others, in~\cite{BrHa05, Cioa05, CiGr07, CiGH09, KrSu06, OCi10} for regular graphs by means of the second or the third largest adjacency eigenvalues and in~\cite{BrHa05, GL22} from Laplacian eigenvalues. As shown in Proposition~\ref{prop:ex}, every graph is a weakly $(n,\beta)$-graph, where $\beta = \frac{\mu_n - \mu_2}{\mu_n + \mu_2}$. Thus Theorems~\ref{main1} and~\ref{main5} and Corollary~\ref{fracpm} imply the following results originally discovered in~\cite{BrHa05,GL22,XueZS}.

\begin{corollary}[\cite{GL22}]\label{cor:matnum}
For any graph $G$ with $n$ vertices, $\displaystyle \alpha'(G)\ge \min\left\{\Big\lceil\frac{\mu_2}{\mu_n} (n -1)\Big\rceil,\ \ \Big\lceil\frac{1}{2}(n-1)\Big\rceil \right\}.$
\end{corollary}

\begin{corollary}
Let $G$ be a nonempty graph on $n$ vertices that satisfies $2\mu_2 \ge \mu_n$.
\\{\em (i) (\cite{BrHa05})} If $n$ is even, then $G$ has a perfect matching.
\\{\em (ii) (\cite{GL22})} If $n$ is odd, then $G$ is factor-critical.
\\ {\em (iii) (\cite{XueZS})} Then $G$ has a fractional perfect matching.
\end{corollary}

Proposition~\ref{prop:ex} also indicates that if $2\delta \ge \mu_2 +\mu_n$, then $G$ is an $(n,\beta)$-graph where $\beta = \frac{\mu_n - \mu_2}{\mu_n + \mu_2}$. Thus Theorems~\ref{main2} and~\ref{main5} imply the following results.
\begin{corollary}\label{cor:lap}
If $2\delta \ge \mu_2 +\mu_n$, then
$$\alpha'(G)\ge \min\left\{\left\lceil\frac{3\mu_2 +\mu_n}{4\mu_n} (n -1)\right\rceil, \left\lceil\frac{1}{2}(n-1)\right\rceil \right\}\ge \left\lceil\frac{n-1}{4}\right\rceil.$$
\end{corollary}

\begin{corollary}
If $2\delta \ge \mu_2 +\mu_n$ and $3\mu_2 \ge \mu_n,$ then $G$ has a perfect matching when $n$ is even and $G$ is factor-critical when $n$ is odd,  and thus $G$ has a fractional perfect matching.
\end{corollary}

Since every $(n, d, \lambda)$-graph is an $(n, \lambda/d)$-graph, Theorems~\ref{main2} implies the following result for regular graphs via adjacency eigenvalues, which provides a valuable addition to the results of~\cite{CiGr07, CiGH09, OCi10} on matching number.
\begin{corollary}
Let $G$ be an $(n, d, \lambda)$-graph. Then
$$\alpha'(G)\ge \min\left\{\left\lceil\left(\frac{3(d-\lambda)}{4(d+\lambda)}+\frac{1}{4}\right) (n -1)\right\rceil, \left\lceil\frac{1}{2}(n-1)\right\rceil \right\}.$$\end{corollary}


Similarly, Theorems~\ref{main3} and~\ref{main4} imply lower bounds on bipartite matching number via Laplacian eigenvalues by setting $\beta =\frac{\mu_n -\mu_2}{\mu_n + \mu_2}$.
\begin{corollary}
Let $G$ be a $(U, W)$-bipartite graph with $|W|\ge t|U|$. If $t\ge 1$, then
\begin{align*}
\alpha'(G)\ge \min\left\{\frac{t\big(4\mu_n\mu_2-(\mu_n-\mu_2)^2\big)}{(\mu_n+\mu_2)^2},1\right\}\cdot |U|.
\end{align*}
\end{corollary}

\begin{corollary}[\cite{AFG19}]
Let $G$ be a $(U, W)$-bipartite graph with $|W|\ge s|U|/(s-2)$, where $s>2$. If $\mu_n\le \frac{\sqrt{s}+1}{\sqrt{s}-1}\mu_2$, then $G$ has a matching that saturates $U$.
\end{corollary}
\begin{proof}
By Proposition~\ref{prop:ex}, $G$ is a weakly $(n,\beta)$-graph where $\beta = \frac{\mu_n - \mu_2}{\mu_n + \mu_2}$. Let $t= s/(s-2)$.
If $\mu_n\le \frac{\sqrt{s}+1}{\sqrt{s}-1}\mu_2$, then $t(1-2\beta^2)\ge 1$, and thus  $G$ has a matching that saturates $U$ by Theorem~\ref{main3}.
\end{proof}

\begin{corollary}
Let $G$ be a $(U, W)$-bipartite graph with $|W|\ge |U|$. If $2\delta \ge \mu_2 +\mu_n$, then $G$ has a matching that saturates $U$.
\end{corollary}
\begin{proof}
By Proposition~\ref{prop:ex}, if $2\delta \ge \mu_2 +\mu_n$, then $G$ is an $(n,\beta)$-graph where $\beta = \frac{\mu_n - \mu_2}{\mu_n + \mu_2}$. Since $1\le 1/\beta^2$, by Theorem~\ref{main4}, $G$ has a matching that saturates $U$.
\end{proof}

\begin{remark}
By Proposition~$\ref{prop:ex}$, every graph with no isolated vertices is an $(n,\frac{\sigma\Delta}{\delta})$-graph and every connected graph is a weakly $(n,\beta)$-graph where $\beta=\frac{\sigma_n-\sigma_2}{\sigma_n+\sigma_2}\cdot \frac{\Delta}{\delta}$.
Thus Theorems~{\rm\ref{main1}, \ref{main2}, \ref{main3}, \ref{main4}} and~{\rm\ref{main5}} imply similar results via normalized Laplacian eigenvalues. We omit these results as one can easily obtain them if needed.
\end{remark}

At the end of this section, we would like to revisit fractional matching number $\alpha'_f(G)$. Since $\alpha'_f(G)\ge\alpha'(G)$, any lower bound on $\alpha'(G)$ automatically becomes a lower bound on $\alpha'_f(G)$. While we were not able to prove stronger results for weakly $(n,\beta)$-graphs, we have the following slight improvement of the result in~\cite{GL22} via Laplacian eigenvalues.

\begin{theorem}\label{newt3}
Let $r$ be a real number with $0< r\le 1/2$. \par\noindent $(i)$ If $\mu_2 \ge r \mu_n,$  then $\alpha'_f(G)\ge rn$; \par\noindent $(ii)$ If $\mu_2 \ge r \mu_n$ and $\alpha'(G)\ne \frac{1}{2}(n-1)$, then $\alpha'(G)\ge rn$.
\end{theorem}
\begin{proof}
If $G$ is a complete graph, then the statements are trivially true. Thus we may assume that $G$ is not a complete graph in the following. By Theorems~\ref{btfo} and~\ref{thm:frac}, we suppose that $S$ is a subset of $V(G)$ such that
$i(G-S)-|S|=n-2\alpha'_f(G)$ to show (i) and $o(G-S)-|S|=n-2\alpha'(G)$ to prove (ii). It suffices to show that
$$i(G-S) -|S|\le (1-2r)n\,\,\,\,\text{and}\,\,\,\, o(G-S) - |S|\le (1-2r)n.$$
Let $c_1=i(G-S)$ and $c_2=o(G-S)$. We use $c$ for both $c_1$ and $c_2$ when the proofs are the same. We may assume that $c\ge |S|+1$, for otherwise if $c\le |S|$, then $c -|S|\le 0\le (1-2r)n$, done.

First, we claim $c\ge 2$. Since $\mu_2 \ge r \mu_n>0$, $G$ is connected. If $|S|=0$, then  $c_1=i(G-S)=0$ and (i) clearly holds. Thus we may assume $|S|\ge 1$ for (i) and so $c_1\ge |S|+1\ge 2$. If $c_2=1$, then $|S|=0$ and thus $n$ is odd for (ii).  By the assumption, $o(G-S)-|S|=n-2\alpha'(G)$ and thus
$\alpha'(G)=\frac{1}{2}(n-1)$, a contradiction. Therefore we have $c_2\ge 2$.


By Proposition~\ref{prop:ex}, every graph is a weakly $(n,\beta)$-graph, where $\beta = \frac{\mu_n - \mu_2}{\mu_n + \mu_2}$. If $|V(G-S)|\ge c+1$, then by \eqref{nex1}, $$0<c-|S|< \frac{3\beta-1}{1+\beta}n=\left(1-\frac{2\mu_2}{\mu_n}\right)n\le (1-2r)n,$$
done. Thus we can assume that $|V(G-S)|=c.$
By taking a vertex in each component of $G-S$, we obtain an independent set of cardinality $c$. It is well known from~\cite{GoNe08,LuLT07} that
\begin{equation}\label{nex2}
c\le \frac{\mu_n-\delta}{\mu_n}n.
\end{equation}
Since $G$ is not a complete graph, we have $\mu_2\le \delta$ by Fiedler~\cite{Fie1}. Combining this with \eqref{nex2}, it follows that
$$0<c-|S|=2c-n\le \left(1-\frac{2\delta}{\mu_n}\right)n \le \left(1-\frac{2\mu_2}{\mu_n}\right)n\le (1-2r)n,$$ as required.
\end{proof}

\begin{remark}
With the similar argument as in Theorem {\em\ref{newt3}}, we have a refinement of Theorem {\em\ref{main2}}: If $G$ is an $(n,\beta)$-graph and $\alpha'(G)\ne\frac{1}{2}(n-1)$, then $\alpha'(G)\ge \min\left\{\frac{2-\beta}{2(1+\beta)},\, \frac{1}{2}\right\}\cdot n.$
\end{remark}

\section{Toughness  and Hamilton cycle}\label{sect:toughness}

 In this section, we study toughness and Hamilton cycle of (weakly) $(n,\beta)$-graphs.

\subsection{Toughness}
The {\bf toughness} $t(G)$ of a connected graph $G$ is defined as $t(G)=\min\left\{\frac{|S|}{c(G-S)}\right\}$, in which the minimum is taken over all proper subsets $S\subset V(G)$ such that $G-S$ is disconnected and $c(G-S)$ denotes the number of components of $G-S$. By convention, the toughness of a complete graph is infinity. For any real number $r\ge 0$, $G$ is {\bf $r$-tough} if $t(G)\ge r$. The graph toughness was introduced by Chv\'atal \cite{Chva73} in 1973 to study cycle structures.

Toughness of regular graphs from eigenvalues of the adjacency matrix was first studied by Alon~\cite{Alon95} who proved that for any connected $d$-regular graph $G$, $t(G)>\frac{1}{3}\left(\frac{d^2}{d\lambda+\lambda^2}-1\right)$. By this result, Alon showed that for every $t$ and $g$ there are $t$-tough graphs of girth strictly greater than $g$, strengthening a result of Bauer, Van den Heuvel and Schmeichel \cite{BaVS95} 
thus   disprove in a strong sense a conjecture of Chv\'atal \cite{Chva73} that there exists a constant $t_0$ such that every $t_0$-tough graph is pancyclic. Brouwer~\cite{Brou95} independently showed that $t(G)>\frac{d}{\lambda}-2$ for any connected $d$-regular graph $G$, and he also conjectured that $t(G) \ge\frac{d}{\lambda}-1$ in \cite{Brou95, Brou96}. Some related results can be found in~\cite{ CiGu16,CiWo14, Gu21, POHBWWC}. Recently Brouwer's conjecture has been confirmed in~\cite{Gu21b}.

\begin{theorem}[\cite{Gu21b}]\label{thm:gu}
For any connected $d$-regular graph $G$, $\displaystyle t(G)\ge\frac{d}{\lambda}-1$.
\end{theorem}

Here we use a similar idea and generalize the result to $(n,\beta)$-graphs. By Proposition~\ref{prop:ex}, every $(n,d,\lambda)$-graph is an $(n,\lambda/d)$-graph, and thus Theorem~\ref{thm:gu} is a corollary of the following theorem.
\begin{theorem}\label{tough-main}
Let $G$ be a connected $(n,\beta)$-graph. Then $\displaystyle t(G)\ge \frac{1-\beta}{\beta}.$
\end{theorem}
\begin{proof}
We can assume that $\beta <1$, for otherwise it is trivially true. Suppose to the contrary that
\begin{equation}\label{eq:contr}
t(G) < \frac{1-\beta}{\beta}.
\end{equation}
By definition, suppose that $S$ is a subset of $V(G)$ such that $\frac{|S|}{c(G-S)}=t(G)$. Let $B = V(G-S)$. Denote $c(G-S)=c$ and $t(G)=t$. Then $|S| = tc$ and so $|B| = n - tc$.

By Lemma~\ref{L23l}, $$c\le \frac{\beta n}{1 + \beta}.$$

If $|B|\le \frac{2\beta n}{1 +\beta}$, then $|S| = n-|B|\ge\frac{(1-\beta)n}{1 +\beta}$, and so
$$t= |S|/c\ge \frac{(1-\beta)n}{1 +\beta}\cdot \frac{1+\beta}{\beta n}=\frac{1-\beta}{\beta},$$
contrary to~\eqref{eq:contr}. Thus, we may assume that
\begin{equation*}
|B| > \frac{2\beta n}{1 + \beta},
\end{equation*}
and thus $|B|\ge 2c+1$.

Let $V_1,V_2,\ldots, V_c$ denote the vertex sets of the $c$ components of $G-S$. Without loss of generality, we may assume that $|V_1|\le |V_2|\le \cdots\le |V_c|$. If $\sum_{i=1}^{c-1} |V_i| = c-1$, then each $V_i$ is a single vertex for $i=1,2,\ldots, c-1$. By Lemma~\ref{L23l}, $$t=\frac{|S|}{c}>\frac{1}{\beta(1+\beta)} >\frac{1-\beta}{\beta},$$
contrary to \eqref{eq:contr}. Thus we may assume that $\sum_{i=1}^{c-1} |V_i|\ge c$.





\medskip
For $\sum_{i=1}^{c-1} |V_i|\ge c$, it was proved in \cite{Gu21b} that when $|B|\ge 2c+1$, $V_1,V_2,\ldots, V_c$ can be partitioned into two disjoint sets $X$ and $Y$ such that $e(X, Y)=0$, and  $|Y|\ge |X|\ge c$.

By the definition of $(n, \beta)$-graphs and Lemma \ref{22l},
$$t=\frac{|S|}{c}\ge\frac{(1-\beta)|X|}{c\beta} \ge\frac{1-\beta}{\beta},$$
completing the proof.
\end{proof}




%

Most recently, Haemers~\cite{Haem20} conjectured that $t(G)\ge \frac{\mu_2}{\mu_n -\delta}$ for any graph $G$, and a partial result has been obtained in~\cite{GH21} that $t(G)\ge \frac{\mu_2}{\mu_n - \mu_2}$. By Proposition~\ref{prop:ex}, if $2\delta \ge \mu_2 +\mu_n$ for a graph $G$, then $G$ is an $(n,\beta)$-graph, where $\beta = \frac{\mu_n - \mu_2}{\mu_n + \mu_2}$. Thus we have the following corollary.
\begin{corollary}
If $2\delta \ge \mu_2 +\mu_n$ for a connected graph $G$, then
$\displaystyle t(G)\ge \frac{2\mu_2}{\mu_n - \mu_2}$.
\end{corollary}

By Proposition~\ref{prop:ex} again, every graph with no isolated vertices is an $(n,\beta)$-graph, where $\beta=\sigma\Delta/\delta$.
We have the following lower bound on toughness via normalized Laplacian eigenvalue $\sigma$. Notice that $\sigma = \lambda/d$ for $d$-regular graphs and thus the following corollary extends
Theorem~\ref{thm:gu} to general graphs.
\begin{corollary}
For any connected graph $G$, $\displaystyle t(G)\ge\frac{\delta}{\sigma\Delta}-1$.
\end{corollary}

Many graph properties are related to toughness, and we refer readers to the survey~\cite{BaBS06} for more toughness related results. Using similar ideas as~\cite{GH21}, we have various applications on, for instances, matching extensions, factors, $k$-walk, and spanning trees of bounded maximum degree, etc. These results can be easily obtained when needed, and thus are omitted here.

\medskip

Here is a variation of toughness introduced by Enomoto~\cite{Enom98} to study $k$-factors. Let $t'(G)$ be the minimum of $\frac{|S|}{c(G-S)-1}$ taken over all proper subsets $S\subset V(G)$ such that $G-S$ is disconnected. Clearly $t'(G) > t(G)$. We have a quick proof of the following result for weakly $(n,\beta)$-graphs.

\begin{theorem}\label{weakly-tough}
For a connected weakly $(n,\beta)$-graph $G$, $\displaystyle t'(G) >\frac{1-\beta}{2\beta}$ and $\displaystyle t(G) >\frac{5(1-\beta)}{11\beta}$.  Furthermore, if $n$ is large enough, then $\displaystyle t(G) >\left(\frac{1}{2}-\varepsilon\right)\frac{1-\beta}{\beta}$ for any $\varepsilon >0$.
\end{theorem}
\begin{proof}
By definition, suppose that $S$ is a subset of $V(G)$ such that $\frac{|S|}{c(G-S)}=t(G)$ or $\frac{|S|}{c(G-S)-1}=t'(G)$. Denote $c(G-S)=c$. By Lemma~\ref{L23l}, $$\text{ $|S|>\frac{(c-1)(1-\beta)}{2\beta}$, and so $t'(G)=\frac{|S|}{c-1}> \frac{1-\beta}{2\beta}$.}$$

To show the lower bounds on $t(G)$, let $V_1,V_2,\ldots, V_c$ be the vertex sets of the $c$ components of $G-S$. Without loss of generality, we may assume that $|V_1|\le |V_2|\le \cdots\le |V_c|$. Clearly we may suppose that $0<\beta<1$. In the proof of Lemma~\ref{L23l}, it is proved that if either $|V(G)-S|\ge c+1$ or $c$ is even, then  $V_1,V_2,\ldots, V_c$ can be partitioned into two disjoint parts $X$ and $Y$ such that $e(X,Y)=0$ and $|Y|\ge |X|\ge \frac{c}{2}$. By \eqref{ssiz},
$$t(G)=\frac{|S|}{c} \ge \frac{1 -\beta}{\beta}\cdot  \frac{|X|}{c}\ge \frac{1 -\beta}{2\beta}.$$
Therefore it suffices to assume that $c$ is odd and $|V(G)-S|=c\ge 3$ in the following.

If $t(G)=\frac{|S|}{c}=\frac{n-c}{c}\le  \frac{5(1-\beta)}{11\beta}$, then $\beta\le \frac{5c}{11n-6c}$. Let $X=\bigcup_{j=1}^{\lfloor c/2\rfloor} V_j$ and $Y=V(G)-S-X$. By \eqref{Liu3e}, $$\frac{c^2-1}{(2n+1-c)(2n-c-1)}=\frac{|X|\cdot |Y|}{(n-|X|)(n-|Y|}\le \beta^2\le \left(\frac{5c}{11n-6c}\right)^2,$$
which implies that $(21c^2 - 121)n -11 c(c^2 - 1)\le 0$, contrary to the fact that $n\ge c+1\ge 4$.

It remains to prove that $t(G) >\left(\frac{1}{2}-\varepsilon\right)\frac{1-\beta}{\beta}$ when $n$ is large enough. Let $X=\bigcup_{j=1}^{\lfloor c/2\rfloor} V_j$ and $Y=V(G)-S-X$. Then $|Y|=\frac{c+1}{2} > |X| =\frac{c-1}{2}\ge \left(\frac{1}{2}-\varepsilon\right)c$ when $c\ge \frac{1}{2\varepsilon}$. By \eqref{ssiz} again,
$$t(G)=\frac{|S|}{c} > \frac{1 -\beta}{\beta}\cdot \frac{|X|}{c}\ge \left(\frac{1}{2}-\varepsilon\right)\frac{1-\beta}{\beta}.$$
If $c <\frac{1}{2\varepsilon}$, then $$t(G)=\frac{|S|}{c}=\frac{n-c}{c}=\frac{n}{c}-1 >2\varepsilon n -1 \ge \left(\frac{1}{2}-\varepsilon\right)\frac{1-\beta}{\beta}$$
when $2\varepsilon n\ge \left(\frac{1}{2}-\varepsilon\right)\frac{1-\beta}{\beta} +1$.
This completes the proof.
\end{proof}
\begin{remark}
 In the above proof, the requirement $2\varepsilon n\ge \left(\frac{1}{2}-\varepsilon\right)\frac{1-\beta}{\beta} +1$ is not optimized. It probably can be improved with a little more effort, but we did not pursue it. The same idea also works to prove some specific bound like $t(G) > \frac{k}{2k+1}\cdot \frac{1 -\beta}{\beta}$ for larger $n$. For instance,  it is true that $t(G) >\frac{6(1-\beta)}{13\beta}$ when $n\ge 6$.
\end{remark}



To end this subsection, we mention another variation of toughness. The {\bf scattering number $s(G)$} is defined by Jung~\cite{Jung78} in the ``additive dual'' sense of toughness that $s(G) = \max\left\{c(G-S)-|S|\right\}$ taken over all proper subsets $S\subset V(G)$ such that $G-S$ is disconnected. A graph is an $H$-free graph if it does not contain $H$ as an induced subgraph. The scattering number is closely related to disjoint paths, Hamilton paths and Hamilton cycles in $P_4$-free graphs  which are also called $D^*$-graphs in~\cite{Jung78}. Lemma~\ref{L23l} implies that $s(G)\le \max\left\{\frac{2\beta-1}{\beta+1}n, 0\right\}$ for $(n,\beta)$-graphs and $s(G)\le \max\left\{\frac{(3\beta-1)n +2(1-\beta)}{\beta+1},0\right\}$ for weakly $(n,\beta)$-graphs. 

\subsection{Hamilton cycle}
A graph is {\bf Hamiltonian} if it contains a Hamilton cycle. The following conjecture of Krivelevich and Sudakov~\cite{KrSu03} is well known for pseudorandom graphs. A stronger conjecture was then made by the first author via Laplacian eigenvalues.

\begin{conjecture}[Krivelevich and Sudakov~\cite{KrSu03}]\label{conj:ks}
For any $(n,d,\lambda)$-graph $G$ with large enough $n$, there exists a constant $K>0$ such that if  $d/\lambda > K$, then $G$ is Hamiltonian.
\end{conjecture}

\begin{conjecture}[Gu, Conjecture~5.11 in~\cite{GH21}]\label{conj:gu}
There exists a positive constant $C<1$ such that if $\mu_2/ \mu_n \ge C$ and $n\ge 3$ {\em(}or large enough $n${\em)}, then $G$ is Hamiltonian.
\end{conjecture}

Now we may make a similar conjecture for weakly $(n,\beta)$-graphs as below. Note that any $(n,\beta)$-graph is also a weakly $(n,\beta)$-graph. By Proposition~\ref{prop:ex}, Conjecture~\ref{conj:beta} is stronger than both Conjectures~\ref{conj:ks} and~\ref{conj:gu}.
\begin{conjecture}\label{conj:beta}
There exists a constant $\beta >0$ such that every weakly $(n,\beta)$-graph with $n\ge 3$ {\em(}or large enough $n${\em)} is Hamiltonian.
\end{conjecture}

Chv\'atal~\cite{Chva73} conjectured that there exists a constant $t_0$ such that every $t_0$-tough graph is Hamiltonian. By Theorem~\ref{weakly-tough}, it is not hard to see that Chv\'atal's conjecture implies Conjecture~\ref{conj:beta}. Chv\'atal's conjecture is still open, however, it has been verified for many graph families, especially for graphs with various forbidden subgraphs.  For instance, it was proved by Jung~\cite{Jung78} that any $P_4$-free graph $G$ with at least $3$ vertices is Hamiltonian if $t(G)\ge 1$. By Theorem~\ref{tough-main}, this implies that Conjecture~\ref{conj:beta} is true with $\beta\le 1/2$ for $P_4$-free graphs. It has also been shown that Chv\'atal's conjecture is true for planar graphs, chordal graphs, $2K_2$-free graphs, $(P_2\cup P_3)$-free graphs, $(P_2\cup 3P_1)$-free graphs, $(P_3\cup 2P_1)$-free graphs, etc., where $P_n$ is a path on $n$ vertices. Thus Conjectures~\ref{conj:ks},~\ref{conj:gu} and~\ref{conj:beta} are true for these families of graphs.




\section{A bipartite analogue}\label{sect:bipart}

Motivated by the bipartite analogue of pseudorandom graphs by Thomason~\cite{Thom89}, we have the following bipartite analogue of $(n,\beta)$-graphs.

A $(U, W)$-bipartite graph $G$ is called a {\bf $(U, W, \beta)$-bipartite graph if }
$$\frac{|X||Y|}{(|U|-|X|)(|W|-|Y|)}\le \beta^2$$ holds for every pair of vertex subsets $X\subseteq U$ and $Y\subseteq W$ such that $e(X, Y)=0$.
Clearly a $(U,W)$-bipartite graph  is a $(U,W,0)$-bipartite graph if and only if it is a complete bipartite graph.

Here is an example.
It is proved by Butler~\cite[Lemma~38]{Bu08} that $$\frac{\vol(X)\vol(Y)}{\vol(U\setminus X) \vol(W\setminus Y)}\le (1-\sigma_2)^2$$
for any connected $(U, W)$-bipartite graph and every pair of vertex subsets $X\subseteq U$ and $Y\subseteq W$ such that $e(X, Y)=0$. Let $\delta_1, \delta_2$ be the minimum degree of vertices in $U$ and $W$, respectively. Similarly, let $\Delta_1$ and $\Delta_2$ be the maximum degree of $U$ and $W$, respectively. We have $\delta_1 |X|\le \vol(X)$, $\delta_2 |Y|\le \vol(Y)$, $\vol(U\setminus X)\le\Delta_1 (|U|-|X|)$ and $\vol(W\setminus Y)\le\Delta_2 (|W|-|Y|)$, which implies that
$$\frac{|X||Y|}{(|U|-|X|)(|W|-|Y|)}\le \left((1-\sigma_2)\cdot \sqrt{\frac{\Delta_1\Delta_2}{\delta_1\delta_2}}\right)^2
\le \left((1-\sigma_2)\cdot\frac{\Delta}{\delta}\right)^2.$$
Thus any connected $(U,W)$-bipartite graph is a $(U, W, \beta)$-bipartite graph for $\beta\ge (1-\sigma_2)\cdot \sqrt{\frac{\Delta_1\Delta_2}{\delta_1\delta_2}}$.

\begin{theorem}
Let  $G$ be a connected $(U, W, \beta)$-bipartite graph with $|W|\ge t|U|$. If $t\ge 1$ and $\beta>0$, then $$\alpha'(G)\ge \min\left\{t/\beta^2,1\right\}\cdot |U|.$$
\end{theorem}
\begin{proof}
Let $r=\min\left\{t/\beta^2,1\right\}$, $|U|=n_1$ and $|W|=n_2$. Then  $n_2\ge t n_1$. By Theorem~\ref{hall}, it suffices to show that for every $S\subseteq U$, $|S|-|N(S)|\le (1-r)n_1$.

This is trivial if either $|S|\le (1-r)n_1$ or $|S|\le |N(S)|$ since $r\le 1$. Thus we may assume that $|S|>(1-r)n_1$ and $|S|>|N(S)|$.

By contradiction, we suppose that there exists an $S\subseteq U$ such that $|S|-|N(S)| > (1-r)n_1$. Let $|S|=x$ and $|N(S)|=y$, and so $x-y> (1-r)n_1$. Choose $X=S$ and $Y=W\setminus N(S)$. Then there is no edge between $X$ and $Y$. By definition, we have
$$\frac{x(n_2-y)}{(n_1-x)y}=\frac{|X||Y|}{(n_1-|X|)(n_2-|Y|)}\le \beta^2.$$
Since $y<x$ and $n_2\ge tn_1$, it follows that
$$\frac{tx}{y}<\frac{x(n_2-y)}{(n_1-x)y}\le \beta^2,$$
and so $y> tx/\beta^2$.

Since $r\le 1$, we have $$0\le (1-r)n_1<x-y< x -\frac{tx}{\beta^2} = \left(1 -\frac{t}{\beta^2}\right)x\le \left(1 -\frac{t}{\beta^2}\right)n_1,$$
which implies that $r> t/\beta^2$, a contradiction. This completes the proof.
\end{proof}

\begin{corollary}
Let $G$ be a connected  $(U, W)$-bipartite graph with $|W|\ge t|U|$. If  $t\ge 1$, then
$$\alpha'(G)\ge \min\left\{
\frac{1}{(1-\sigma_2)^2}\cdot \frac{\delta_1\delta_2}{\Delta_1\Delta_2}\cdot t,
\, \, 1\right\}\cdot |U|.$$
\end{corollary}

\par\medskip
\noindent
{\bf Acknowledgements.}
\\The authors would like to thank Dr. Shaowei Sun for discussing normalized Laplacian eigenvalues. Gu is partially supported by a grant from the Simons Foundation (522728), and Liu is partially supported by
  Natural Science Foundation of Guangdong Province (No. 2022A1515011786).


\end{document}